\documentclass[11pt,a4paper,english]{article}
\usepackage{amssymb}
\usepackage{amsfonts}
\usepackage[all]{xy}
\usepackage{amsfonts,amscd,amssymb,amsthm}
\usepackage{amsfonts,amssymb,amsmath,amsxtra,amscd,amsthm,eucal,graphicx,graphics}
\usepackage{tikz}
\usepackage{enumerate}
 \usepackage[T1]{fontenc}
\usepackage[utf8]{inputenc}
\usepackage{babel}
\newcommand{\overbar}[1]{\mkern 1.5mu\overline{\mkern-1.5mu#1\mkern-1.5mu}\mkern 1.5mu}
\newcommand{\cupdot}{\mathbin{\mathaccent\cdot\cup}}

\theoremstyle{definition}
\newtheorem{theorem}{Theorem}[section]

 \title{On the real roots of $\sigma$-Polynomials}
    \author{Jason Brown\footnote{Corresponding author, e-mail: jason.brown@dal.ca} \and Aysel Erey}
    \date{Department of Mathematics and Statistics\\ Dalhousie University \\ Halifax, Nova Scotia, Canada B3H 3J5 \\[\baselineskip] \today }

\begin{document}

\maketitle

\begin{abstract}
The $\sigma$-polynomial is given by $\sigma(G,x) = \sum_{i=\chi(G)}^{n} a_{i}(G)\, x^{i}$,
where $a_{i}(G)$ is the number of partitions of the vertices of $G$ into $i$ nonempty independent sets. These polynomials are closely related to chromatic polynomials, as the chromatic polynomial of $G$ is given by $\sum_{i=\chi(G)}^{n} a_{i}(G)\, x(x-1) \cdots (x-(i-1))$. It is known that the closure of the real roots of chromatic polynomials is precisely $\{0,~1\} \bigcup [32/27,\infty)$, with $(-\infty,0)$, $(0,1)$ and $(1,32/27)$ being maximal zero-free intervals for roots of chromatic polynomials. We ask here whether such maximal zero-free intervals exist for $\sigma$-polynomials, and show that the only such interval is $[0,\infty)$ -- that is, the closure of the real roots of $\sigma$-polynomials is $(-\infty,0]$. 
\end{abstract}


\section{Introduction}
Given a (finite undirected) graph $G$, the well known chromatic polynomial $\pi(G,x)$ is, for each nonnegative integer $x=k$, the number of ways to properly colour the vertices of $G$ with the colours $1,2,\ldots,k$ so that no two adjacent vertices receive the same colour. That the function is indeed a polynomial in the number of colours, $x$, is a fundamental result, but as a polynomial, it is natural to investigate about chromatic roots, that is, the roots of chromatic polynomials (indeed, the {\em chromatic number} $\chi(G)$ of graph $G$ is the least positive integer $k$ that is \underline{not} a chromatic root of $G$, and  the famous Four Colour Theorem is equivalent to stating that $4$ is never a chromatic root of a planar graph). There are numerous results on the nature and location of chromatic roots (see, for example, \cite{dongbook}). One of the more interesting problems concerns the real roots of chromatic polynomials. It is well known that the chromatic polynomial is a real polynomial whose coefficients alternate in sign, and hence all real chromatic roots lie in $[0,\infty)$. A basic (and easily proved) fact is that $(0,1)$ is free of chromatic roots, and Jackson \cite{jackson} showed $(1,32/27]$ is a zero-free region for chromatic roots. To complete the picture, Thomassen \cite{thomassen} proved that there are real chromatic roots arbitrarily close to every real greater than or equal to $32/27$, so it follows that the closure of the real chromatic roots is precisely $\{0,~1\} \bigcup [32/27,\infty)$.

We now describe a related family of polynomials. 
An $i$\textit{-colour partition} of $G$ is a partition of the vertices of $G$ into $i$ nonempty independent sets (That is, sets that do not contain any edges of $G$). The {\em $\sigma$-polynomial} of $G$ (see \cite{brenti}) is defined as 
\[ \sigma(G,x) = \sum_{i=\chi(G)}^{n} a_{i}(G)\, x^{i}\]
where $a_{i}(G)$ is the number of $i$-colour partitons of $G$.
The coefficients $a_i$ are also known as the \textit{graphical Stirling numbers}
\cite{duncan,galvin}, as if a graph has no edges, then $a_i$ is simply equal to the Stirling number of the second kind $S(n,i)$. It follows that the $\sigma$-polynomials of empty graphs correspond to the generating functions for Stirling numbers of the second kind and such polynomials were studied \cite{harper,lieb} and shown to have all real roots. 

The $\sigma$-polynomials first arose in the study of chromatic polynomials, since the chromatic polynomial of $G$ can be given as  
\[ \pi(G,x) = \sum a_{i}(G)\, (x)_{\downarrow i},\] 
where $(x)_{\downarrow i} = x(x-1) \cdots (x-i+1)$ is the {\em falling factorial of $x$}. 
The $\sigma$-polynomial was introduced by Korfhage \cite{korfhage} in a slightly different form (he refers to the polynomial $ (\sum_{i=\chi(G)}^n a_i x^i )/ x^{ \chi(G)}$ as the $\sigma$-polynomial), and $\sigma$-polynomials have attracted considerable attention in the literature.

Brenti \cite{brenti} studied the $\sigma$-polynomials extensively and investigated both log-concavity and the nature of the roots. Brenti, Royle and Wagner \cite{royle} proved that a variety of conditions are sufficient for a $\sigma$-polynomial to have only real roots. But a natural question to ask is whether, like chromatic polynomials, there are zero-free regions on the real axis, and moreover, what is the closure of the real $\sigma$-roots (that is, the closure of the roots of $\sigma$-polynomials). A plot of the $\sigma$-roots for all graphs of small order (see Figure~\ref{order7sigmarootsfigure}) seems to indicate that the real roots are filling up the negative real axis, and indeed we shall prove in the next section that this is the case. 

\begin{figure}[htb]
\begin{center}
\includegraphics[scale=1.0]{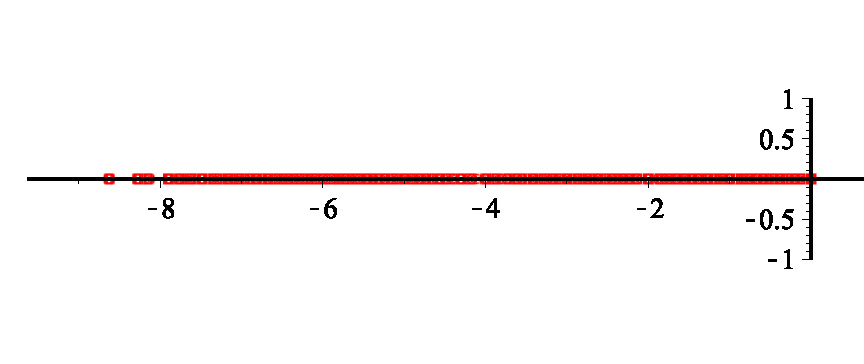}
\caption{The roots of $\sigma$-polynomials of all connected $7$-vertex graphs.}
\label{order7sigmarootsfigure}
\end{center}
\end{figure}


\section{Closure of the real roots of $\sigma$-polynomials}

We will refer to few other classical graph polynomials in our quest to determine the closure of the real $\sigma$-roots.
Given a graph $G$ of order $n$, the \textit{adjacency matrix of G}, $A(G)$, is the $n\times n$ matrix with $(i,j)$-entry equal to $1$ if the $i$-th vertex of $G$ is adjacent to the $j$-th, and equal to $0$ otherwise. The \textit{characteristic polynomial $\phi(G,x)$} of $G$ is defined by
$$\phi(G,x)=\operatorname{det}(xI-A(G)).$$ It is known (see, \cite[pg.~2]{godsilbook}) that if $G$ is a forest then $\phi(G,x)=m(G,x)$, where $m(G,x)$ is the \textit{matching polynomial of G}:
\[ m(G,x)=\sum_{i\geq 0}m_i(G)(-1)^ix^{n-2i},\]
where $m_i(G)$ is the number of matchings of size $i$ in $G$ ($m_0(G)\equiv 1$ by convention).
A famous result (see \cite{godsilbook}) that matching polynomials have all real roots. Also, it is not difficult to see that if $G$ is a triangle-free graph then  $$\sigma(\overbar{G},-x^2)=(-x)^n\, m(G,x).$$

For a sequence $\{f_n(x)\}$  of polynomials, $z$ is called a \textit{limit of roots} of   $\{f_n(x)\}$ if there is a sequence $\{z_n\}$ such that $f_n(z_n)=0$ and $z_n$ converges to $z$.
Let $P_0(x),P_1(x),\dots$ be a sequence of polynomials in $\mathbb{C}[x]$ which satisfy the recursion of degree $k$
\begin{eqnarray}\label{recursion}
P_{n+k}(x)=-\sum_{j=1}^kf_j(x)P_{n+k-j}(x)
\end{eqnarray}
where the $f_j$ are polynomials. The \textit{characteristic equation} of this recursion is 
\begin{eqnarray}\label{chareqn}
\lambda^k+\sum_{j=1}^kf_j(x)\lambda^{k-j}=0.
\end{eqnarray}
Let $\lambda_1(x),\dots ,\lambda_k(x)$ be the roots of the characteristic equation. If the $\lambda_j(x)$ are distinct, then it is well known that the solution of the recursion in \eqref{recursion} has the form
\begin{eqnarray}\label{recursionsoln}
P_n(x)=\sum_{j=1}^k\alpha_j(x)\lambda_j(x)^n.
\end{eqnarray}
(If there are repeated roots values at $x$, \eqref{recursionsoln} is modified in the natural way -- see \cite{beraha}). The $\alpha_j(x)$ are determined in any event by solving the $k$ linear equations in the $\alpha_j(x)$ obtained by letting $n=0,1,\dots, k-1$ in \eqref{recursionsoln} or its variant.
The \textit{nondegeneracy conditions} on the recursive family of polynomials in \eqref{recursion} is defined as follows
\begin{itemize}
\item[(i)] $\{P_n\}$ does not satisfy a recursion of order less than $k$.
\item[(ii)] For no pair $i\neq j$ is $\lambda_i(x)\equiv \omega \lambda_j(x)$ for some $\omega \in \mathbb{C}$ of unit modulus.
\end{itemize}

The following theorem provides the limits of roots of recursive families of polynomials, and will be central to our argument.

\begin{theorem}[Beraha-Kahane-Weiss Theorem \cite{beraha}] Suppose that  $\{P_n(x)\}$ is a sequence of polynomials which  satisfies \eqref{recursion} and the nondegeneracy requirements, and has solution
\[ P_n(x) =\sum_{j=1}^k\alpha_j(x)\lambda_j(x)^n.\]
Then $z$ is a limit of roots of if and only if one of the following holds:
\begin{itemize}
\item[(i)] two or more of the $\lambda_i(x)$ are of equal modulus and strictly greater in modulus than the others (if any).
\item[(ii)] for some $j$, $\lambda_j(x)$ has modulus strictly greater than all the other $\lambda_i(x)$ have and $\alpha_j(z)=0$.
\end{itemize}
\end{theorem}

We are now ready to prove our main result.

\begin{theorem}
The roots of $\sigma$-polynomials are dense in $(-\infty,0)$.
\end{theorem}
\begin{proof}
Since every triangle-free graph $G$ satisfies $\sigma(\overbar{G},-x^2)=(-x)^n\, m(G,x)$, it suffices to show that the roots of matching polynomials of triangle-free graphs are dense in $(-\infty,\infty)$. Obviously, trees are triangle-free graphs. Also, as we already mentioned, the matching polynomials of trees and equal to their characteristic polynomials \cite{godsilbook}. So it suffices to show that the roots of characteristic polynomials of a certain family of trees are dense in $(-\infty,\infty)$.

A tree $T$ is called a \textit{balanced rooted tree with a root w} if $T$ has a vertex $w$ and there exist integers $n_1,n_2,\dots ,n_k$ such that $V(T)\setminus \{w\}$ can be partitioned into $k$ subsets $A_1,A_2,\dots ,A_k$ where $w$ has exactly $n_k$ neighbours in $A_k$ and each vertex in $A_i$ has exactly $n_{i-1}$ neighbours in $A_{i-1}$ for $2\leq i\leq k$. We will denote $T$ by $T(n_k,n_{k-1},\dots, n_1)$. Note that if $n_i=n$ for every $i$ with $1\leq i\leq k$ then $T(n_{k},n_{k-1},\dots, n_1)$ is a complete $n$-ary tree on $\frac{n^{k+1}-1}{n-1}$ vertices and we denote it by $T_n^k$.
In \cite{hic} it was shown that roots of the characteristic polynomial of $T(n_k,n_{k-1},\dots n_1)$ contain the roots of the recursively defined polynomials $P_k(x)$, where
\begin{equation}\label{balancedtreerecursion}
P_j(x)=x\, P_{j-1}(x)-n_j\, P_{j-2}(x)
\end{equation}
for $j=2,\dots,k$; as these roots are roots of characteristic polynomials of graphs, they are always real.

Now let $n$ be a fixed positive integer. By the formula given in \eqref{balancedtreerecursion}, the roots of the characteristic polynomial of $T_n^k$ contain the  roots of the polynomials $P_k(x)$, which are defined recursively as
$$P_k(x)=x\, P_{k-1}(x)-n\,P_{k-2}(x)$$
for $k\geq 2$. The sequence of polynomials $P_2(x),P_3(x),\dots$ satisfying a recursion of degree $2$, with the characteristic equation of this recursion being 
$$\lambda^2-x\lambda+n=0.$$
The roots of this characteristic equation are
$$\lambda_1(x)=\frac{x+\sqrt{x^2-4n}}{2} \quad \text{and} \quad \lambda_2(x)=\frac{x-\sqrt{x^2-4n}}{2} .$$ It is easy to check that the nondegeneracy conditions are satisfied. Therefore, by Beraha-Kahane-Weiss Theorem, every $x\in \mathbb{C}$ satisfying
\begin{equation}\label{sigmadenseprooffirsteqn}
\left| \frac{x+\sqrt{x^2-4n}}{2}\right|= \left|\frac{x-\sqrt{x^2-4n}}{2}\right|
\end{equation}
is a limit of the roots of the sequence $P_2(z),P_3(z),\ldots$. Observe that $x\in \mathbb{C}$ satisfies \eqref{sigmadenseprooffirsteqn} if and only if 
$$\left| x+\sqrt{x^2-4n}\right|= \left|x-\sqrt{x^2-4n}\right|$$
or equivalently,
\begin{equation}\label{sigmadenseproofsecondeqn}
\left| 1+ \frac{\sqrt{x^2-4n}}{x}\right|= \left|1-\frac{\sqrt{x^2-4n}}{x}\right|.
\end{equation}
Now, $x\in \mathbb{C}$ is a solution of \eqref{sigmadenseproofsecondeqn} if and only if $\frac{\sqrt{x^2-4n}}{x}$ is equidistant from $1$ and $-1$, that is, $\frac{\sqrt{x^2-4n}}{x}$ is on the imaginary axis. For $a\in \mathbb{R}$,
 $$\frac{\sqrt{x^2-4n}}{x}=ai \quad \Longleftrightarrow \quad \frac{x^2-4n}{x^2}=-a^2 \quad \Longleftrightarrow \quad  x=\pm 2\sqrt{\frac{n}{1+a^2}}.$$ 
 Thus, $x=\pm 2\sqrt{\frac{n}{1+a^2}}$ is a limit of the roots of $\{P_k(z)\}_{k=2,3,\dots}$ for every real number $a$. Therefore, the roots of $\{P_k(z)\}_{k=2,3,\dots}$ are dense in $[-2\sqrt{n},2\sqrt{n}]$. Now, letting $n\rightarrow \infty$, we obtain that the roots of the characteristic polynomials of complete $k$-ary trees are dense in $\mathbb{R}$.
\end{proof}

\section{Further Remarks}

Of course, our results from the previous section show that the real $\sigma$-roots are unbounded, but one might wonder how far to the left they might go, as a function of the order. From the results of \cite{zhao} (the {\em adjoint polynomial} of a graph is the $\sigma$-polynomial of its complement) it follows that deleting an edge form a graph can only decrease the minimum real $\sigma$-root leftward. Thus the minimum real $\sigma$-root among all graphs of order $n$ occurs for $\overline{K_{n}}$, whose $\sigma$-polynomial is the generating function for the Stirling numbers of the second kind. Weak asymptotics \cite{elbert} have been investigated for the distribution of roots of the scaled $\sigma$-polynomials of such graphs, $\sigma(\overline{K_{n}},nx)$, and we can use the results to show that for every $\varepsilon > 0$ and large enough $n$, there is a real $\sigma$-root of $\overline{K_{n}}$ in the interval $(-ne,-n(1-\varepsilon)e)$ (we omit the details).

\begin{figure}[t]
\begin{center}
\includegraphics[scale=0.75]{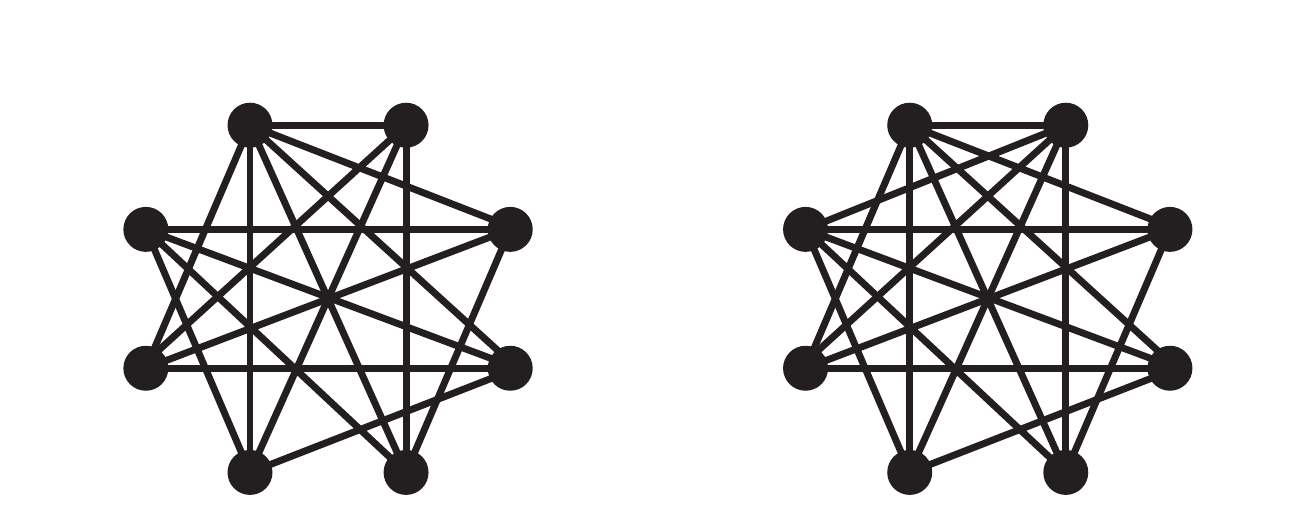}
\caption{The graphs of order $8$ whose $\sigma$-polynomials have nonreal roots.}
\label{smallorder}
\end{center}
\end{figure} 

Have we captured the closure of {\em all} the $\sigma$-roots with our main result? Unfortunately, the answer is no, as there are graphs with nonreal $\sigma$-roots, but they are rare for small graphs -- there are only two such graphs, shown in Figure~\ref{smallorder}, of order at most $8$ \cite{royle}. The situation is not much better for order $9$; among all $274,668$ connected graphs of order $9$, only $42$ have nonreal $\sigma$-roots.
It is not hard to see that the $\sigma$-roots of the join of two disjoint graphs (formed from the disjoint union by adding in all edges between the two graphs) is the union of the $\sigma$-roots of each, so we can use the join operation to produce infinitely many graphs with nonreal $\sigma$-roots, but the set of $\sigma$-roots is rather small (and the roots only have small imaginary part in absolute value). In fact, currently only finitely many nonreal $\sigma$-roots are known. We present a recursive construction of a family of graphs which appear to have distinct nonreal $\sigma$-roots, with larger imaginary part in absolute value.

Let $H_{n,k}^t$ be the graph $K_n$ with a path of size $t$ hanging off $k$ vertices of the clique $K_n$. More precisely, $H_{n,k}^t$ is the graph on $n+kt$ vertices and ${n \choose 2}+kt$ edges whose vertex set is equal to
$$\{u_i : 1\leq i \leq n\}\cupdot \{v_i^j : 1\leq i \leq t, \,  1\leq j \leq k\} $$
and edge set is equal to
$$\{u_iu_j: \, 1\leq i<j\leq n \} \, \cupdot \, \{u_iv_1^i: \, 1\leq i\leq k \} \, \cupdot \, \{v_i^jv_{i+1}^j: \, 1\leq i\leq t-1, \, 1\leq j\leq k\}$$
(see Figure~\ref{Hfamily}).

\begin{figure}[htbp]
\begin{center}
\includegraphics[scale=0.3]{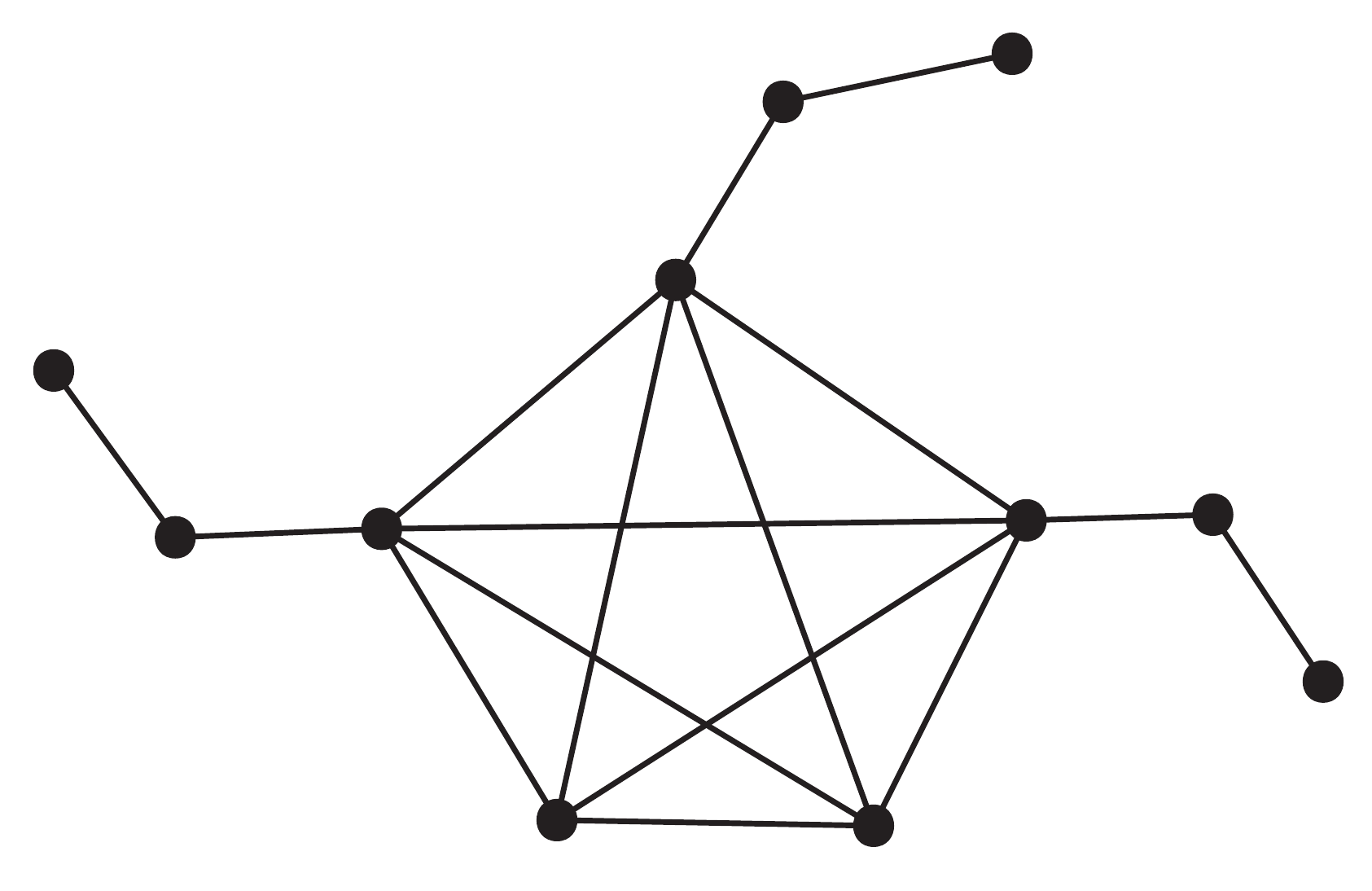}
\caption{The graph $H^2_{5,3}$.}
\label{Hfamily}
\end{center}
\end{figure}

The $\sigma$-roots of such graphs are often nonreal with distinct roots, and can have relatively large imaginary part (see Figures~\ref{nonrealFigureH(n,n,2)},\ref{nonrealFigureH(n,n,3)},\ref{nonrealFigureH(n,n,n)}), although we do not have a proof of this in general. 

\begin{figure}[htbp]
\begin{center}
\includegraphics[scale=1]{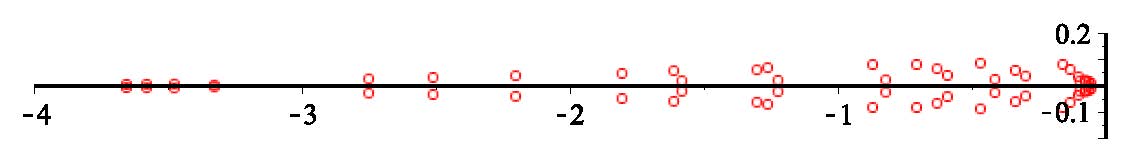}
\caption{Nonreal roots of the adjoint polynomial of $H_{n,n}^n$ for $n=1,\dots , 10$.}
\label{nonrealFigureH(n,n,n)}
\end{center}
\end{figure}

\begin{figure}[htbp]
\begin{center}
\includegraphics[scale=0.5]{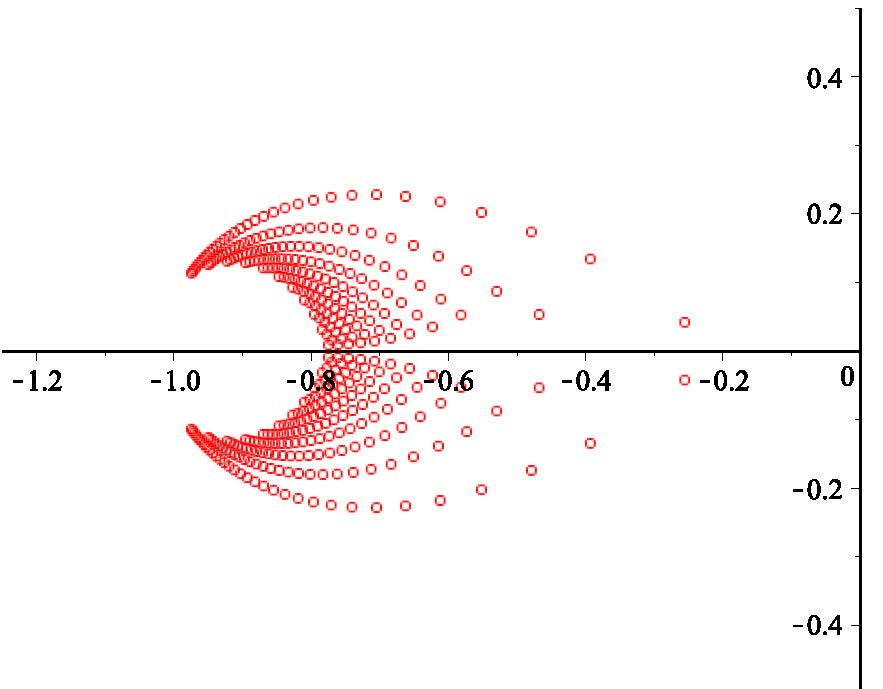}
\caption{Nonreal roots of the adjoint polynomial of $H_{n,n}^2$ for $n=1,\dots , 35$.}
\label{nonrealFigureH(n,n,2)}
\end{center}
\end{figure}

\begin{figure}[htbp]
\begin{center}
\includegraphics[scale=0.5]{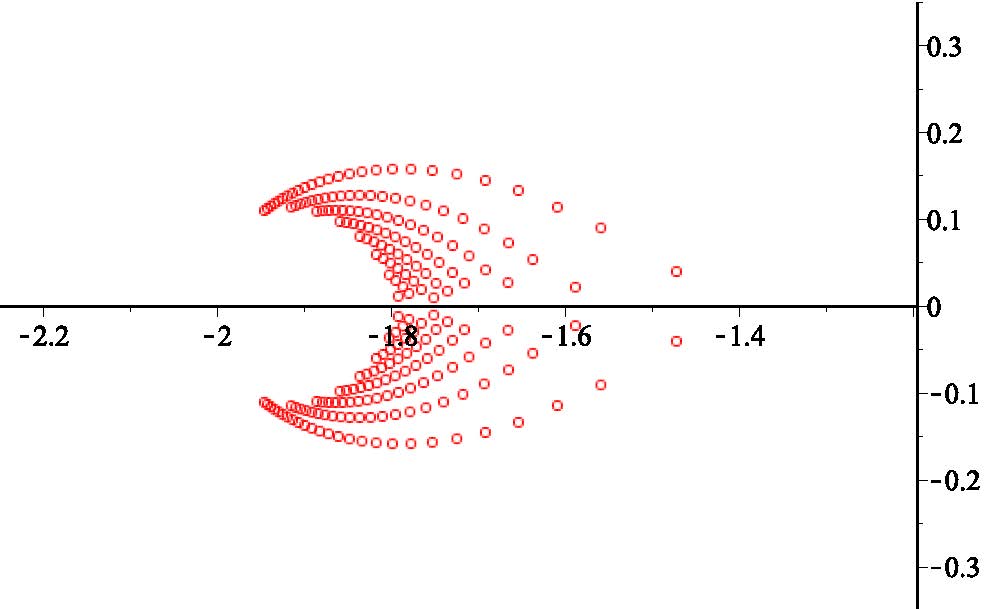}
\caption{Nonreal roots of the adjoint polynomial of $H_{n,n}^3$ for $n=1,\dots , 30$.}
\label{nonrealFigureH(n,n,3)}
\end{center}
\end{figure}

We leave off with three tantalizing questions. First, can the imaginary part of a $\sigma$-root be unbounded? Secondly, can there be sigma roots in the right-half plane, or is the polynomial always {\em Hurwitz stable}? And more generally, what is the closure of the $\sigma$-roots in the complex plane? 

\vskip0.4in
\noindent {\bf \large Acknowledgments:} 
This research was partially supported by a grant from NSERC.

\bibliographystyle{elsarticle-num}
\bibliography{<your-bib-database>}

\end{document}